\documentclass[twoside,11pt]{article}

\usepackage{jmlr2e}

\usepackage{amsmath,amssymb,graphicx,color,subfigure}
\let\hat\widehat
\let\tilde\widetilde

\newenvironment{enum}{
\begin{enumerate}
  \setlength{\itemsep}{1pt}
  \setlength{\parskip}{0pt}
  \setlength{\parsep}{0pt}
}{\end{enumerate}}

\ShortHeadings{Mode Clustering}{Azizyan, Chen, Singh and Wasserman}
\firstpageno{1}

\begin{document}

\title{Risk Bounds For Mode Clustering}

\author{\name Martin Azizyan \email mazizyan@cs.cmu.edu\\
       \addr Machine Learning Department\\
       Carnegie Mellon University\\
       Pittsburgh, PA 15213, USA
       \AND
       \name Yen-Chi Chen \email yenchic@andrew.cmu.edu\\
       \addr Department of Statistics\\
       Carnegie Mellon University\\
       Pittsburgh, PA 15213, USA
       \AND
       \name Aarti Singh \email aarti@cs.cmu.edu\\
       \addr Machine Learning Department\\
       Carnegie Mellon University\\
       Pittsburgh, PA 15213, USA
       \AND
       \name Larry Wasserman \email larry@stat.cmu.edu \\
       \addr Machine Learning Department and Department of Statistics\\
       Carnegie Mellon University\\
       Pittsburgh, PA 15213, USA}

\editor{}

\maketitle

\begin{abstract}
Density mode clustering is a nonparametric 
clustering method.
The clusters are the basins of attraction of the modes of a density estimator.
We study the risk of mode-based clustering.
We show that the clustering risk over the cluster cores --- the regions where the density is high ---
is very small even in high dimensions.
And under a low noise condition,
the overall cluster risk is small even beyond the cores, in high dimensions.
\end{abstract}


\section{Introduction}

Density mode clustering is a nonparametric method
for using density estimation to find clusters
\citep{cheng1995mean,Comaniciu, Arias-Castro2013, Chacon2012}.
The basic idea is to estimate the modes of the density,
and then assign points to the modes by finding the basins of attraction
of the modes. 
See Figures \ref{fig::stylized} and
\ref{fig::meanshiftpicture}.

In this paper we 
study the risk of density mode clustering.
We define the risk in terms of how pairs of points
are clustered under the true density versus the estimated density.
We show that the cluster risk over the cluster cores ---
the high density portion of the basins ---
is exponentially small, independently of dimension.
Moreover, if a certain low noise
assumption holds
then the cluster risk outside the cluster cores is small.
The low noise assumption is similar in spirit to
the Tsyabakov low noise condition that often appears in the 
high dimensional classification literature
\citep{Audibert2007}.

It is worth expanding on this last point.
Because mode clustering requires density estimation ---
and because density estimation is difficult in high dimensions ---
one might get the impression that mode clustering will not work well
in high-dimensions.
But we show that this is not the case.
Even in high dimensions the clustering risk can be very small.
Again, the situation is analogous to classification:
poor estimates of the regression function can still lead to accurate classifiers.

There are many different types of clustering ---
$k$-means, spectral, convex, hierarchical ---
and we are not claiming that mode clustering is necessarily
superior to other clustering methods.
Indeed, which method is best is very problem specific.
Rather, our goal is simply to find bounds on 
the performance of mode base clustering.
Our analysis covers both the low and high-dimensional cases.

{\em Outline.}
In Section 2 we review mode clustering.
In Section 3 we discuss the estimation of the clusters
using kernel density estimators.
Section 4 contains the main results.
After some preliminaries,
we bound the risk over the cluster cores in Section 4.3.
In Section 4.4 we bound the risk outside the cores under a low noise assumption.
In Section 4.5 we consider the case of Gaussian clusters.
In Section 4.6 we show a different method
to bound the risk in the low dimensional case.
Section 5 contains some numerical experiments.
We conclude with a discussion in Section 6.

\vspace{1cm}

{\em Related Work.}
Mode clustering is usually implemented
using the mean-shift
algorithm which is
discussed in \cite{Fukunaga,cheng1995mean,Comaniciu}.
The algorithm is analyzed in \cite{Arias-Castro2013}.
\cite{li2007nonparametric, Azzalini:2007la} introduced mode
clustering to the statistics literature. The related idea of clustering based on
high density regions was proposed in \cite{Hartigan1975}.
\cite{Chacon2011} and \cite{Chacon2012} propose several methods for selecting the
bandwidth for estimating the derivatives of the
density estimator which can in turn be used as a bandwidth selection
rule for mode clustering.
A method that is related to mode clustering is
clustering based on trees constructed from density level sets.
See, for example, 
\cite{chaudhuri2010rates},
\cite{kpotufe2011pruning} and
\cite{kent2013debacl}.

{\em Notation:}
We let $p$ denote a density function,
$g$ its gradient and $H$ its Hessian.
A point $x$ is a {\em local mode} 
(i.e. a local maximum)
of $p$ if
$||g(x)||=0$ and all the eigenvalues
of $H(x)$ are negative.
Here, $||\cdot ||$ denotes the usual $L_2$ norm.
In general, the eigenvalues of
a symmetric matrix $A$ are denoted by
$\lambda_1\geq \lambda_2\geq \cdots$.
We write $a_n \preceq b_n$ to mean that there is some $C>0$
such that $a_n \leq C b_n$ for all large $n$.
We use $B(x,\epsilon)$ to denote a closed ball of radius
$\epsilon$ centered at $x$.
The boundary of a set $A$ is denoted by $\partial A$.

\begin{figure}
\begin{center}
\begin{tabular}{ccc}
\includegraphics[scale=.35]{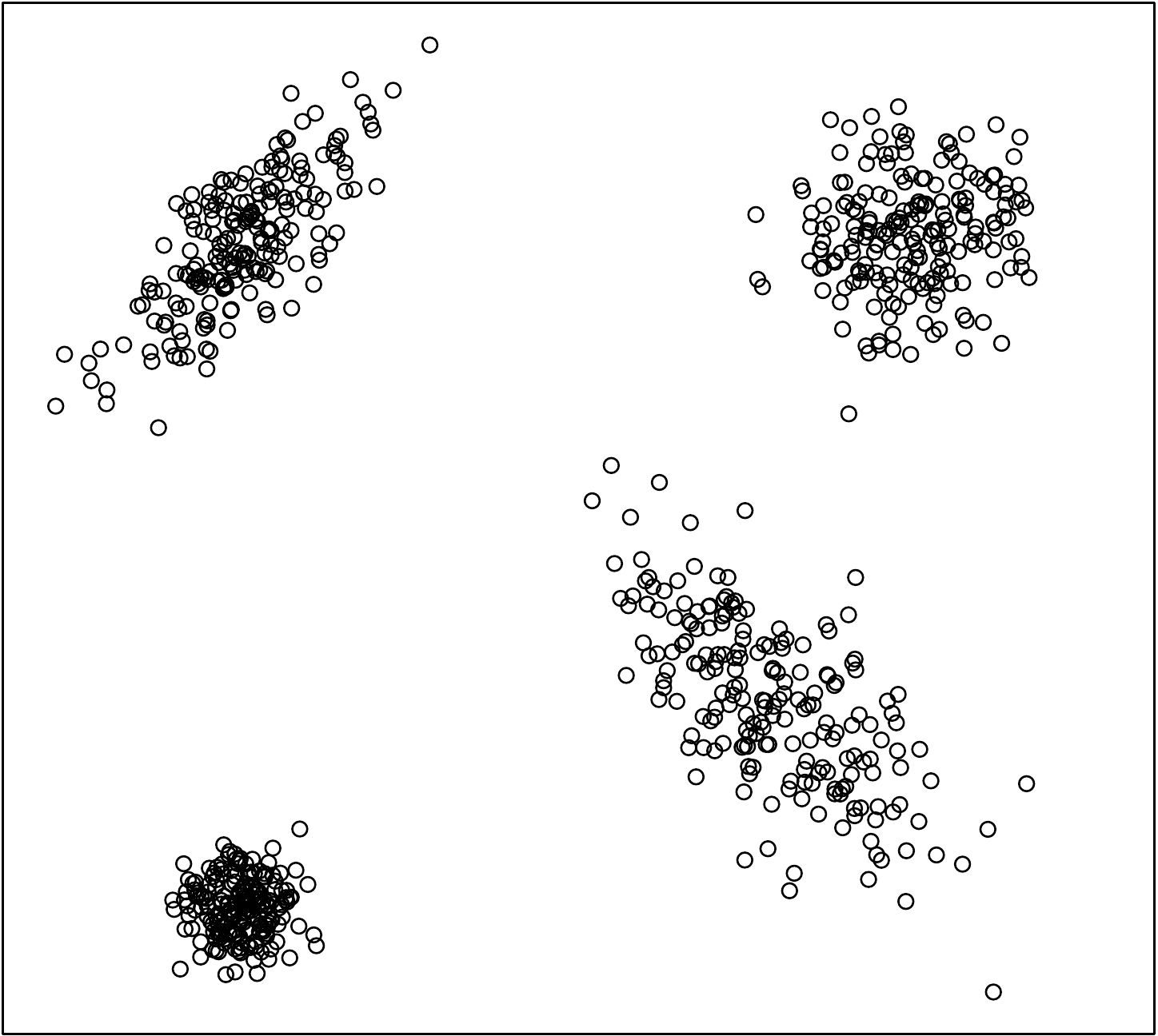} & 
\includegraphics[scale=.35]{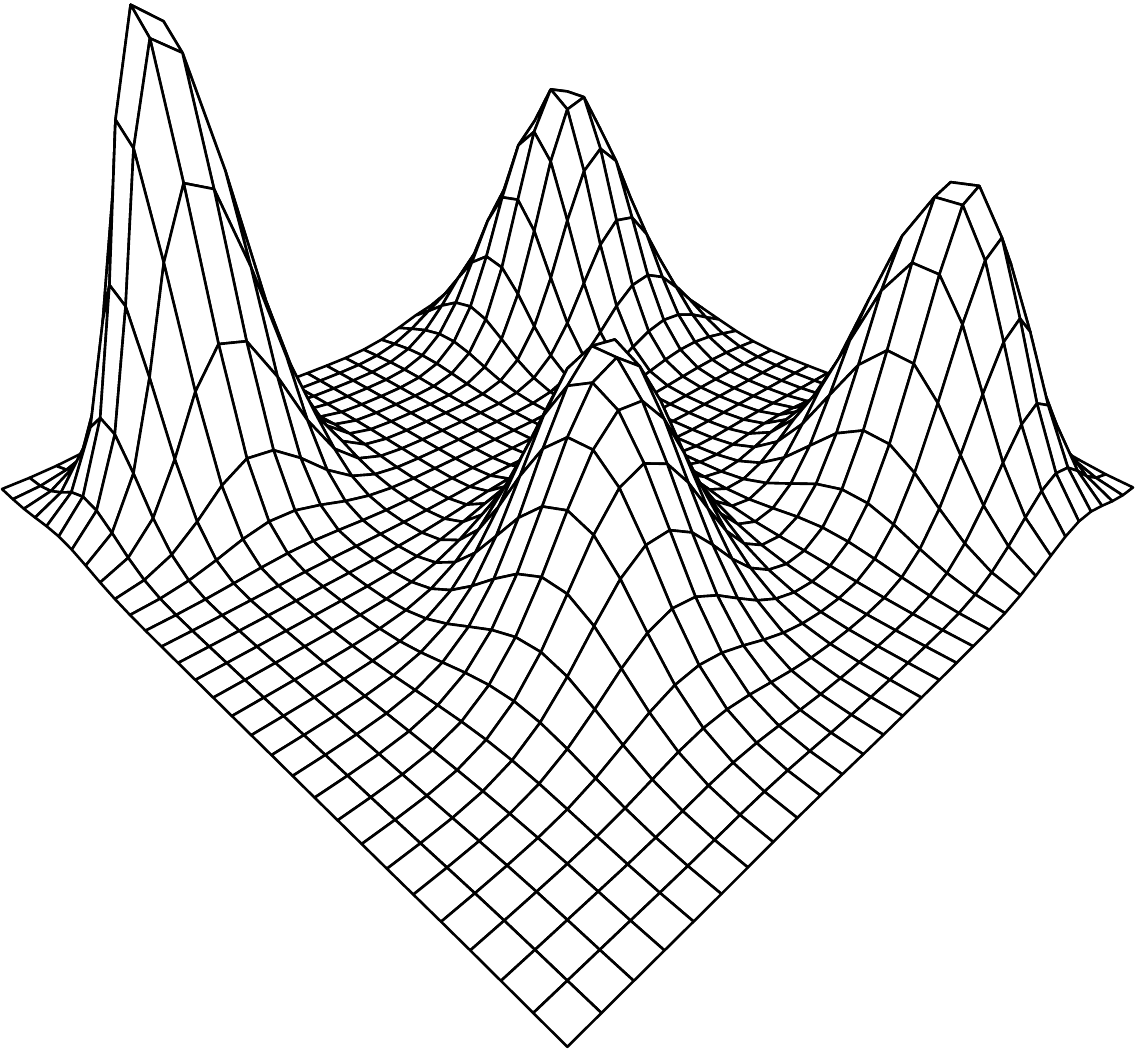} & 
\includegraphics[scale=.35]{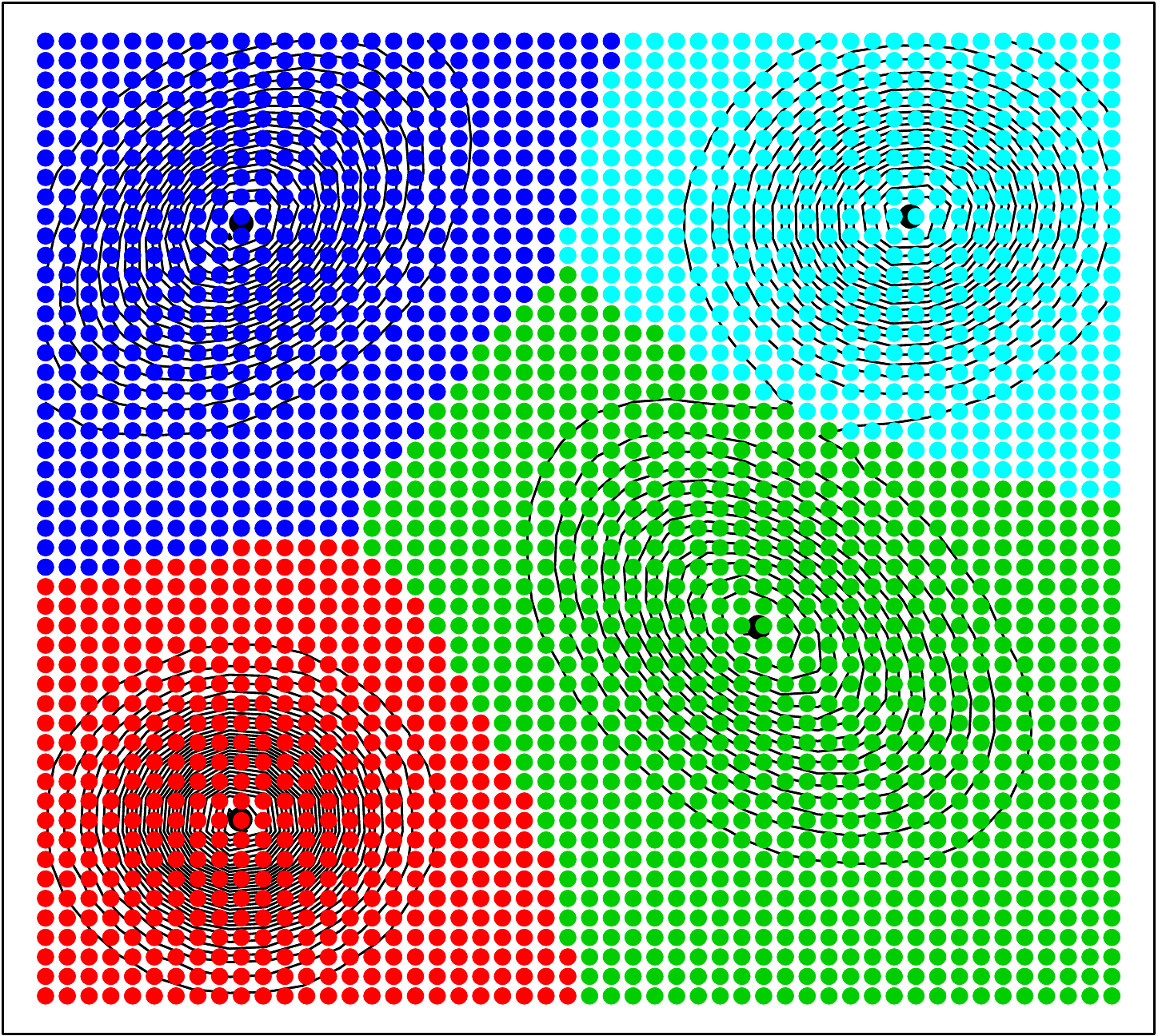} 
\end{tabular}
\end{center}
\caption{\em Left: a simple dataset.
Middle: the kernel density estimator.
Right: The four estimated modes and their basins of attractions.}
\label{fig::stylized}
\end{figure}

\section{Mode Clustering and Morse Theory}

Here we give a brief review of mode clustering, also called mean-shift clustering;
more details can be found in Cheng (1995), Comaniciu and Meer (2002),
Arias-Castro, Mason, Pelletier (2014) and  Chacon (2012).

\subsection{Morse Theory}

We will need some terminology from Morse theory.
Good references on Morse theory
include
\cite{edelsbrunner2010computational,
milnor1963morse,
matsumoto2002introduction,
banyaga2004morse}.

Let $p$ be a bounded continuous density on $\mathbb{R}^d$
with gradient $g$ and Hessian $H$.
A point $x$ is a {\em critical point} if
$||g(x)||=0$.
We then call $p(x)$ a {\em critical value}.
A point that is not a critical point is a {\em regular point}.

The function $p$ is
a {\em Morse function}
if all its critical values are non-degenerate
(i.e. the Hessian at each critical point is non-singular).
A critical point $x$ is a mode, or local maximum,
if the Hessian $H(x)$ is negative definite at $x$.
The {\em index} of a critical point $x$
is the number of negative eigenvalues of $H(x)$.
Critical points are maxima, minima or saddlepoints.

The {\em flow}
starting at $x$ is the path
$\pi_x: \mathbb{R}\to \mathbb{R}^d$
satisfying $\pi_x(0) =x$ and
\begin{equation}\label{eq::flow}
\pi_x'(t) = \nabla p(\pi_x(t)).
\end{equation}
The flow $\pi_x(t)$ defines the direction of
steepest ascent at $x$.
The {\em destination} and {\em origin} of the flow $\pi_x$ are defined by
\begin{equation}
{\rm dest}(x) = \lim_{t\to\infty}\pi_x(t),\ \ \ 
{\rm org}(x) = \lim_{t\to -\infty}\pi_t(x).
\end{equation}
If $x$ is a critical point, then ${\rm dest}(x)=x$.

The {\em stable manifold} corresponding to a critical point $y$---
also called
the {\em descending manifold} or the {\em basin of attraction}---
is
\begin{equation}
{\cal C}(y) = \Bigl\{x:\ {\rm dest}(x)=y \Bigr\}.
\end{equation}
In particular,
the basin of attraction of a mode $m$ is called a {\em cluster}.
See Figures
\ref{fig::DM} and \ref{Fig::ex_D}.

\begin{figure}
\begin{center}
\includegraphics[scale=.75]{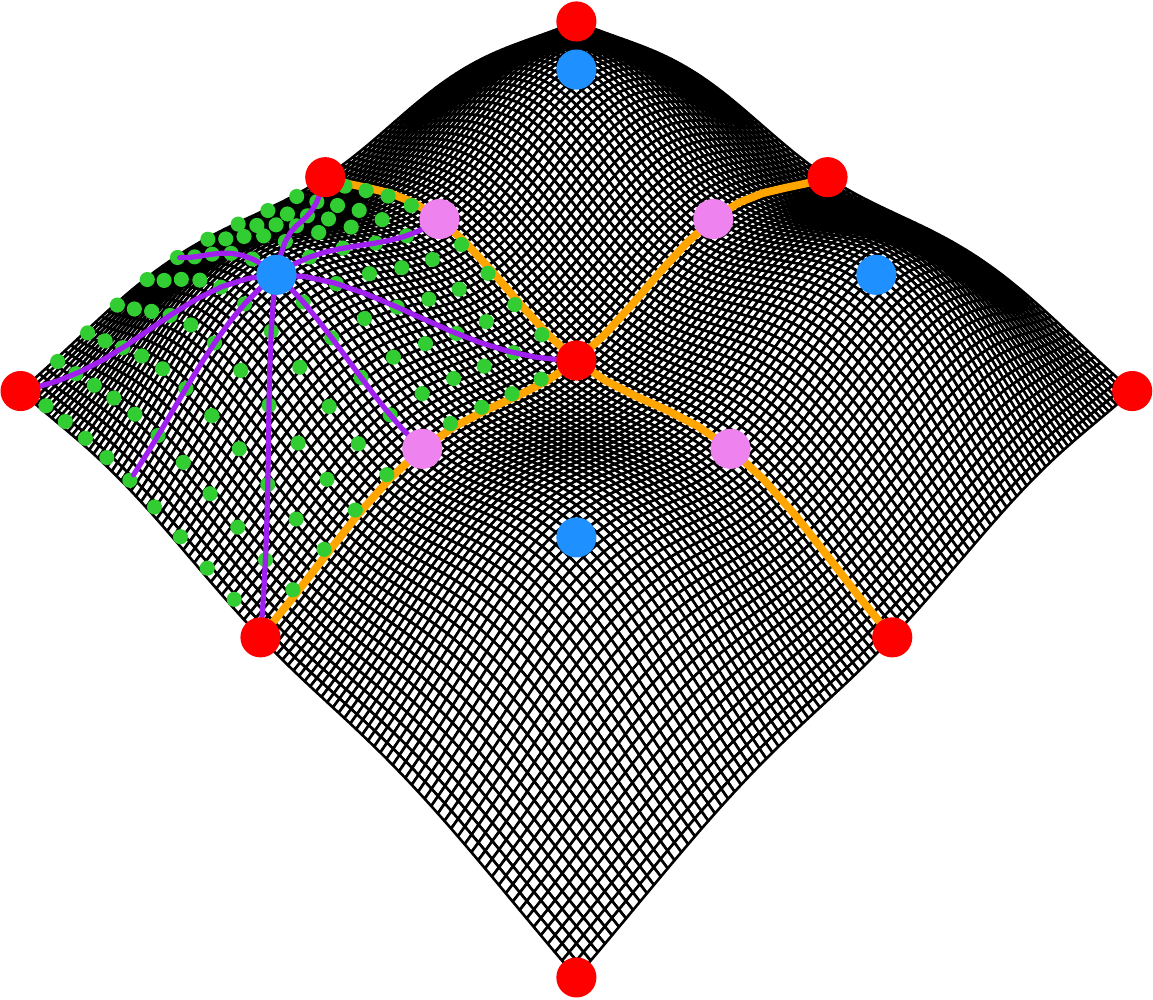}
\end{center}
\caption{\em A Morse function with four modes.
Each solid blue dot is a mode.
Each red dot is a minimum.
Pink dots denote saddle points.
The green area is the descending manifold (cluster) for one of the modes.}
\label{fig::DM}
\end{figure}



\begin{figure}
\center
\includegraphics{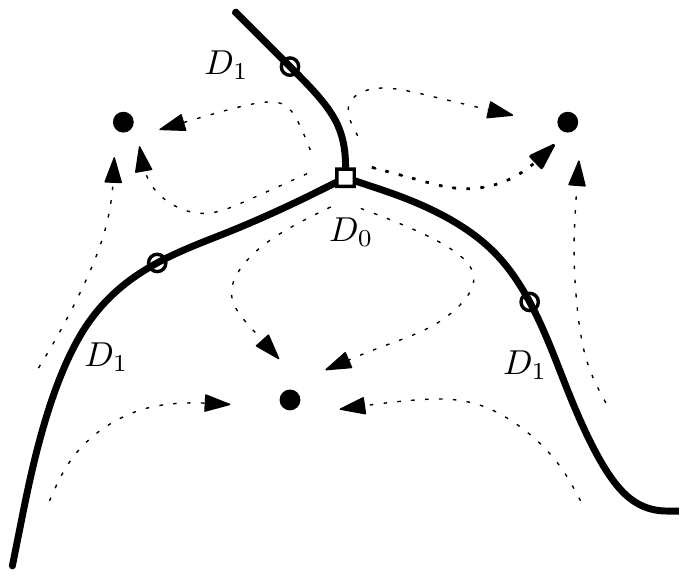}
\caption{\em The three large black dots are the three local modes
that induce three clusters based on the corresponding basins of attraction.
The cluster boundaries, $D$, consists of the local minima (the square box, $D_0$) and 
the three thick smooth curves are $D_1$. The circles on the boundaries are saddle points.
The dotted lines show the flow lines.}
\label{Fig::ex_D}
\end{figure}

Let us mention a few properties of Morse functions that are useful:
\begin{enum}
\item Excluding critical points, two flow lines are either disjoint or they are the same.
\item The origin and destination of a flow line
are critical points (except at boundaries of clusters).
The set of points $x$ whose destinations are
not modes are on the boundaries of clusters and
form a set of measure 0.
\item Flow lines are monotonic: 
$p(x_t)$ is a non-decreasing function of $t$,
where $x_t = \pi_x(t)$.
Further,
$p({\rm dest}(x)) \geq p({\rm org}(x))$ and
${\rm dest}(x) \neq {\rm org}(x)$ if $x$ is a regular point.
\item The index of 
${\rm dest}(x)$ is greater than the index of ${\rm org}(x)$.
\item The flow has the semi-group property:
$\phi(x,t+s) = \phi(\phi(x,t),s)$
where
$\phi(x,t) = \pi_x(t)$.
\item 
Let ${\cal C}$ be the basin of attraction of a mode $m$.
If $y$ is a critical point in the closure of ${\cal C}$ and $y\neq m$,
then $y\in \partial {\cal C}$.
\end{enum}

\subsection{Clusters}

Consider a distribution $P$ on ${\cal K}\subset \mathbb{R}^d$ with density $p$.
We assume that $p$ is a Morse function with finitely many critical points.
The modes of $p$ are denoted by
\begin{equation}
{\cal M} = \{m_1,\ldots, m_k\}
\end{equation}
The corresponding clusters are
${\cal C}_1, \ldots, {\cal C}_k$
where
${\cal C}_j = \Bigl\{x:\ {\rm dest}(x)=m_j\Bigr\}$.
Define the clustering function
$c:{\cal K} \times {\cal K}\to \{0,1\}$ by
$$
c(x,y) =
\begin{cases}
1 & {\rm if\ } {\rm dest}(x) = {\rm dest}(y)\\
0 & {\rm if\ } {\rm dest}(x) \neq {\rm dest}(y).
\end{cases}
$$
Thus, 
$c(x,y)=1$ if and only if $x$ and $y$ are in the same cluster.

Let $X_1,\ldots, X_n \in \mathbb{R}^d$
be random vectors
drawn iid from $P$.
Let $\hat p$ be an estimate of the density $p$
with corresponding estimated modes
$\hat{\cal M}= \{\hat m_1,\ldots, \hat m_\ell\}$,
and basins
$\hat{\cal C} = \{ \hat {\cal C}_1,\ldots, \hat {\cal C}_\ell\}$.
This defines a cluster function $\hat c$.

In this paper,
the {\em pairwise clustering loss} is defined to be
\begin{equation}
L = \frac{1}{\binom{n}{2}}\sum_{j < k} I \Bigl( \hat c(X_j,X_k)\neq c(X_j,X_k)\Bigr)
\end{equation}
which is one minus the Rand index.
The corresponding clustering risk is
$R = \mathbb{E}[L]$.

\section{Estimated Clusters}

Estimating the clusters involves
two steps. First we estimate the density
then we estimate the modes and their basins of attractions.
To estimate the density we use the standard kernel density estimator
\begin{equation}
\hat p_h(x) =\frac{1}{n}\sum_{i=1}^n \frac{1}{h^d}
K\left( \frac{||x-X_i||}{h}\right).
\end{equation}
We will need the following result on 
the accuracy of derivative estimation.
We state the result without proof as it is
a simple generalization of the result in
\cite{Gine2002}
which is based on Talagrand's inequality.
In fact, it is essentially
a different way of stating the results of Lemmas 2 and 3 in
\cite{Arias-Castro2013}.

\begin{lemma}\label{lemma::kernel}
Let $p_h(x) = \mathbb{E}[\hat p_h(x)]$.
Assume that the kernel is Gaussian.
Also assume that $p$ has bounded continuous derivatives up to and including third order.
Then:

(1: Bias) There exist $c_0,c_1,c_2$ such that
$$
\sup_x |p_h(x) - p(x)| \leq c_0 h^2,\ \ \ 
\sup_x ||\nabla p_h(x) - \nabla p(x)|| \leq c_1 h^2,\ \ \ 
\sup_x ||\nabla^2 p_h(x) - \nabla^2 p(x)|| \leq c_2 h.
$$

(2: Variance) There exist $b,b_0, b_1, b_2$ such that, 
if $(\log n/n)^{1/d} \leq h \leq b$ where $b < 1$,
then, 
\begin{align*}
\mathbb{P}(\sup_x |\hat p_h(x) - p_h(x)|> \epsilon) & \leq e^{-b_0 n h^d \epsilon^2}\\
\mathbb{P}(\sup_x ||\hat \nabla p_h(x) - \nabla p_h(x)||> \epsilon) & \leq 
e^{-b_1 n h^{d+2}\epsilon^2}\\
\mathbb{P}(\sup_x ||\hat \nabla^2 p_h(x) - \nabla^2 p_h(x)||> \epsilon) & \leq 
e^{-b_2 n h^{d+4}\epsilon^2}.
\end{align*}
\end{lemma}

{\bf Remark:}
It is not necessary to use a Gaussian kernel.
Any kernel that satisfies the conditions in
\cite{Arias-Castro2013} will do.

To find the modes of $\hat p_h$
we use the well-known mean shift algorithm.
See Figures \ref{fig::meanshift} and
\ref{fig::meanshiftpicture}.
The algorithm approximates the flow defined by
(\ref{eq::flow}).
The algorithm finds the modes, the basins of attractions
and the destination $\hat{\rm dest}(x)$
of any point $x$.
A rigorous analysis of the algorithm can be found in
\cite{Arias-Castro2013}.

\begin{figure}
\fbox{\parbox{6in}{
\begin{center}
{\sc Mean Shift}
\end{center}
\begin{center}
\begin{enum}
\item Choose a set of grid points $G=\{g_1,\ldots, g_N\}$.
Usually, these are taken to be the data points.
\item For each $g\in G$, iterate until convergence:
$$
g^{(r+1)} \longleftarrow \frac{\sum_i X_i K(||g^{(r)}-X_i||/h)}{\sum_i K(||g^{(r)} -X_i||/h)}.
$$
\item 
Let $\hat{\cal M}$ be the unique elements of
$\{g_1^{(\infty)},\ldots, g_N^{(\infty)}\}$.
Output $\{g_1^{(\infty)},\ldots, g_N^{(\infty)}\}$, $\hat{\cal M}$
and $\hat{\rm dest}(g_j) = g_j^{(\infty)}$.
\end{enum}
\end{center}
}}
\caption{The Mean Shift Algorithm}
\label{fig::meanshift}
\end{figure}

\begin{figure}
\begin{center}
\includegraphics[scale=.4]{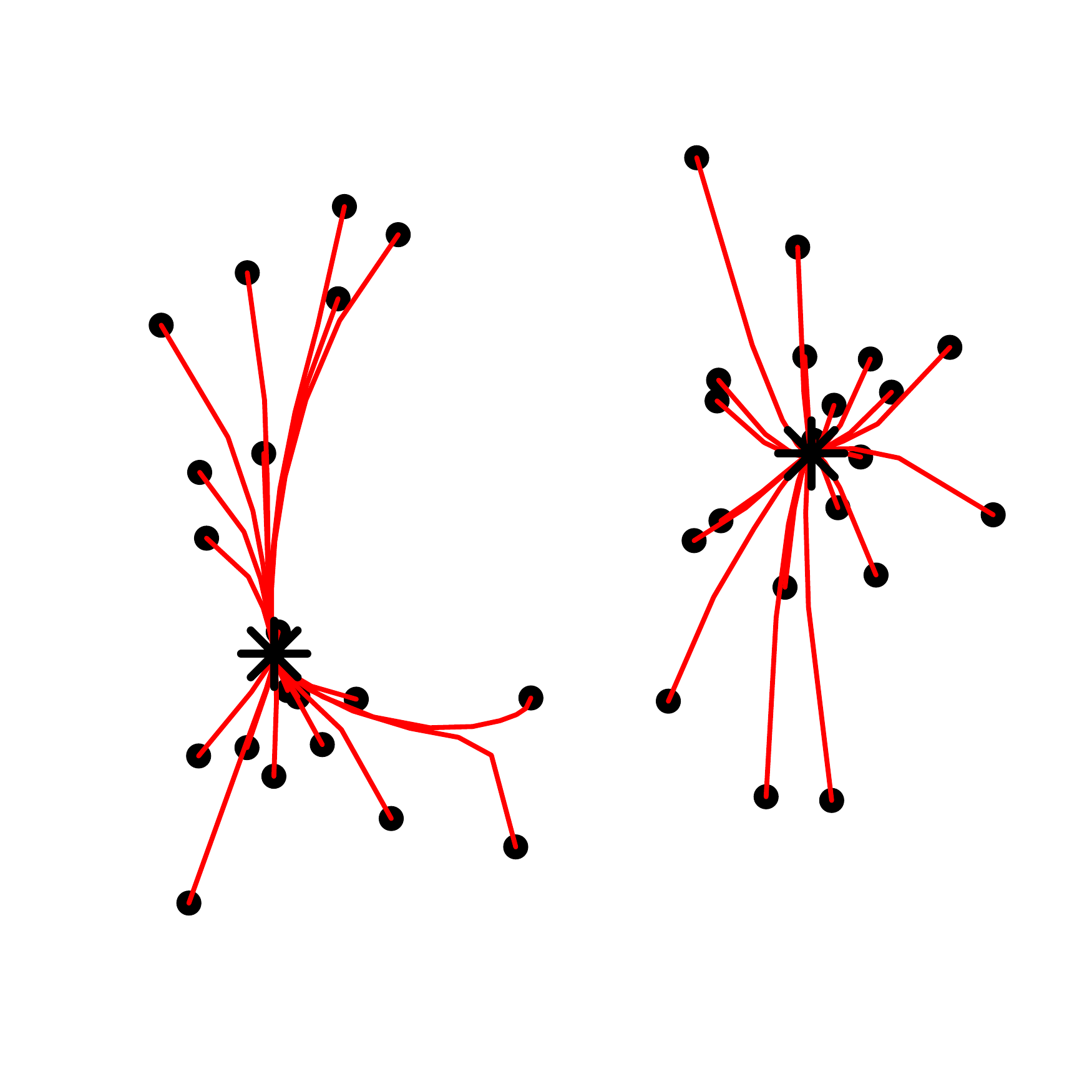}
\end{center}
\vspace{-.5in}
\caption{An illustration of the mean shift algorithm.
The data are moved to the two modes along their gradient ascent paths.}
\label{fig::meanshiftpicture}
\end{figure}

\section{Bounding the Risk}

We are now ready to bound the clustering risk.
We begin by introducing
some preliminary concepts.

\subsection{Stability}

To bound the clustering risk,
we need to control how much the critical points can change
when the density is perturbed.
In particular,
we need the following result which is Lemma 16 from Chazal et al (2015).

\begin{lemma}\label{lemma::morse}
Let $p$ be a density with compact support.
Assume that $p$ is a Morse function with finitely many critical values
$C=\{c_1,\ldots, c_L\}$ and that
$p$ has two continuous derivatives on the interior of its support and non-vanishing
gradient on the boundary of its support.
Let $q$ be another density and let
$\eta = \max\{\eta_0,\eta_1,\eta_2\}$
where
$$
\eta_0 = \sup_x|p(x) - q(x)|,\ 
\eta_1 = \sup_x|| \nabla p(x) - \nabla q(x)||,\ 
\eta_2 = \sup_x|| \nabla^2 p(x) - \nabla^2 q(x)||
$$
where $\nabla^2$ is the vec of the Hessian.
There are constants $\kappa\equiv \kappa(p)$
and $A \equiv A(p)$
such that,
if $\eta\leq \kappa$ then the following is true.
The function $q$ is Morse and has $L$ critical points
$C'=\{c_1',\ldots, c_L'\}$.
After a suitable relabeling of the indices,
$c_j$ and $c_j'$ have the same Morse index for all $j$ and
$\max_j ||c_j - c_j'|| \leq A(p) \eta.$
\end{lemma}

\subsection{The Cluster Cores}

An important part of our analysis involves, what we refer to as,
the cluster cores.
These are the high density regions inside each cluster.
Consider the clusters ${\cal C} =\{ {\cal C}_1,\ldots, {\cal C}_k\}$.
Define
\begin{equation}
\xi_j = \sup_{x\in\partial {\cal C}_j}p(x)
\end{equation}
where
$\partial {\cal C}_j$ is the boundary of
${\cal C}_j$.
For any $a\geq 0$
we define the $j^{\rm th}$ {\em cluster core} by
\begin{equation}
{\cal C}^\dagger_j(a) = \Bigl\{x\in {\cal C}_j:\ p(x) \geq \xi_j + a\Bigr\}.
\end{equation}
See Figure \ref{fig::xi}.

\begin{theorem}\label{thm::core-stuff}
Let $p$ be a density function with compact support.
Assume that $p$ is a Morse function with finitely many critical values
and that
$C_g \equiv \sup_x ||g(x)|| < \infty$
where $g$ is the gradient of $p$.
Let $\tilde p$ be another density and 
define $\eta,\eta_0,\eta_1,\eta_2, A(p)$ and $\kappa(p)$ as in Lemma \ref{lemma::morse}.
Let $\tilde \pi$ denote the paths defined by $\tilde p$.
Let ${\cal C}$ be a cluster of $p$ with mode $m$ and
let $\xi = \sup_{x\in \partial {\cal C}}p(x)$.
Let $a = C_g A \eta + 2 \eta_0$.
Assume that $\eta < \kappa(p)$
and that
\begin{equation}
p(m) > a + A \eta C_g +\xi = 2A \eta C_g + 2\eta_0 + \xi.
\end{equation}
Then the following hold:
\begin{enum}
\item If $x \in {\cal C}^\dagger(a)$ then
$d(x,\partial {\cal C})\geq \frac{a}{C_g}$
where $d(x,A) = \inf_{y\in A}||x-y||$.
\item $B(m,A\eta) \subset {\cal C}^\dagger(a-2\eta_0)$.
\item $\tilde p$ has a mode $\tilde m \in {\cal C}^\dagger(a-2\eta_0)$.
\item $\tilde p$ has no other critical points in 
${\cal C}^\dagger(a-2\eta_0)$.
\item Let $x\in {\cal C}^\dagger(a)$.
Then $\tilde \pi_x(t) \in {\cal C}^\dagger(a-2\eta_0)$
for all $t\geq 0$.
\item Let $x,y\in {\cal C}^\dagger(a)$.
Then ${\rm dest}(x) = {\rm dest}(y) = m$ and
$\tilde{\rm dest}(x) = \tilde{\rm dest}(y) = \tilde{m}$.
Hence,
$c(x,y) = \tilde c (x,y)$.
\end{enum}
\end{theorem}

\begin{proof}

1. Let $z$ be the projection of $x$ onto 
$\partial {\cal C}$.
(Choose any projection if it is not unique.)
Using an exact Taylor expansion,
\begin{align*}
\xi + a & \leq p(x) = p(z) + (x-z)^T \int_0^1 g(z + u(x-z)) \, du\\
& \leq p(z) + C_g ||x-z|| = p(z) + C_g d(x,\partial {\cal C})\\
& \leq \xi + C_g d(x,\partial {\cal C}).
\end{align*}

2. Let $x\in B(m,A\eta)$.
Then
\begin{align*}
p(x) &= p(m) + (x-m)^T \int_0^1 g(m + u (x-m)) du \geq
p(m) - ||x-m|| C_g \geq p(m) - A\eta C_g > a + \xi
\end{align*}
and hence
$x\in {\cal C}^\dagger(a) \subset {\cal C}^\dagger(a-2\eta_0)$.

3. By Lemma 
\ref{lemma::morse},
$\tilde p$ has a mode $\tilde m$ such that
$||m-\tilde m||\leq A \eta$.
The result then follows from part 2.

4. Let $\tilde c$ be a critical point of $\tilde p$
different from $\tilde m$.
By Lemma \ref{lemma::morse},
there is a critical point $c$ of $p$ such that
$||c-\tilde c|| \leq A \eta$.
Now $c$ must be on the boundary of some cluster or must be a minimum.
Either way, it is not in the interior of ${\cal C}$.
Let $r$ be the point on $\partial {\cal C}$ closest to $c$.
Then
$d(c,{\cal C}^\dagger(a-2\eta_0)) \geq d(r,{\cal C}^\dagger(a-2\eta_0))$.
By part 1,
$d(r,{\cal C}^\dagger(a-2\eta_0)) > (a-2\eta_0)/C_g$.
Thus,
$d(\tilde c,{\cal C}^\dagger(a-2\eta_0)) > (a-2\eta_0)/C_g - A \eta$.
By the definition of $a$, it follows then that
$d(\tilde c,{\cal C}^\dagger(a-2\eta_0)) >0$ and hence
$\tilde c\notin {\cal C}^\dagger(a-2\eta_0)$.

5. Let $x\in {\cal C}^\dagger(a)$.
Then, for any $t\geq 0$,
\begin{align*}
p(\tilde \pi_x(t)) & \geq 
\tilde p(\tilde \pi_x(t)) - \eta_0 \geq
\tilde p(\tilde \pi_x(0)) - \eta_0 \\
&=
\tilde p(x) - \eta_0 \geq  p(x) - 2\eta_0 \geq \xi + a - 2\eta_0.
\end{align*}

6. Let $x,y\in {\cal C}^\dagger(a)$.
Trivially, we have that ${\rm dest}(x) = {\rm dest}(y) = m$.
From the previous result,
$\tilde{\rm dest}(x) \in {\cal C}^\dagger(a-2\eta_0)$.
From parts 3 and 4, the only critical point of $\tilde p$ in 
${\cal C}^\dagger(a-2\eta_0)$ is $\tilde m$.
Similarly for $y$.
Hence,
$\tilde{\rm dest}(x) = \tilde{\rm dest}(y) = \tilde{m}$.
\end{proof}

\subsection{Bounding the Risk Over the Cores}

Now we bound the risk for the data points that are in the cluster cores.

\begin{theorem}\label{thm::core-risk}
Assume that $p$ is a Morse function with
finitely many critical values.
Denote the modes and clusters by
$m_1,\ldots, m_k$ and
${\cal C}_1,\ldots, {\cal C}_k$.
Let $\hat p_h$
be the kernel density estimator.
Let $\eta = \max\{\eta_0,\eta_1,\eta_2\}$
where
$$
\eta_0 = \sup_x|\hat p_h(x) - p(x)|,\ 
\eta_1 = \sup_x|| \nabla \hat p_h(x) - \nabla p(x)||,\ 
\eta_2 = \sup_x|| \nabla^2 \hat p_h(x) - \nabla^2 p(x)||.
$$
Let $a = C_g A \eta + 2\eta_0$ and let
${\cal C}^\dagger = \bigcup_j {\cal C}^\dagger_j(a)$ and let
${\cal X} = \{ X_i:\ X_i\in {\cal C}^\dagger(a)\}$
be the points in the cores.
Let $\xi_j = \sup \{p(x):\ x\in \partial {\cal C}_j\}$.
\begin{enum}
\item If
$$
p(m_j) >  2A \eta C_g + 2\eta_0 + \xi_j
$$
for each $j$,
then
$\hat c(X_i,X_j) = c(X_i,X_j)$ for every $X_i,X_j \in {\cal X}$.
\item If $h_n\to 0$ and $n h_n^{d+4} \to \infty$, then
\begin{equation}
\mathbb{P}\Biggl( \hat c(X_i,X_j) \neq c(X_i,X_j)\ \ 
{\rm for\ any\ }\ X_i,X_j \in {\cal X}\Biggr) \leq e^{-n b}
\end{equation}
for some $b>0$ (independent of $d$).
\end{enum}
\end{theorem}

{\bf Remark:}
Note that $\eta, \eta_0,\eta_1,\eta_2$ are functions of $n$ but we suppress the dependence for simplicity.

\begin{proof}
1. From Lemma \ref{lemma::kernel},
we have that
$\mathbb{P}(\eta > \kappa(p))$ is exponentially small.
Hence, Lemma \ref{lemma::morse} applies.
If
$p(m_j) >  2A \eta C_g + 2\eta_0+\xi$ for all $j$,
then Theorem \ref{thm::core-stuff} implies that
$\hat c(X_i,X_j) = c(X_i,X_j)$ for every $X_i,X_j \in {\cal X}$.

2. We need to show that
$p(m_j) >  2A \eta C_g + 2\eta_0+\xi_j$ for all $j$ so we can apply part 1.
The probability that
$p(m_j) >  2A \eta C_g + 2\eta_0+\xi_j$ fails for some $j$,
is $\mathbb{P}(\eta > q)$
where $q>0$ is a constant.
If $h_n\to 0$ and $n h_n^{d+4} \to \infty$, then
from Lemma \ref{lemma::kernel},
$\mathbb{P}(\eta > q)$ is exponentially small:
\begin{align*}
\mathbb{P}(\eta > q) &\leq
\sum_{j=0}^2 \mathbb{P}(\eta_j > q)\\
& \leq
\exp\left(- b_0 n h_n^d (q- c_0 h_n^2)^2\right) + 
\exp\left(- b_1 n h_n^{d+2} (q- c_1 h_n^2)^2\right) + 
\exp\left(- b_2 n h_n^{d+4} (q- c_2 h_n)^2\right)\\
&\leq e^{-n b}
\end{align*}
for some $b>0$.
\end{proof}

\subsection{Beyond the Cores}

Now we bound the risk beyond the cores.
Furthermore, we explicitly let $d=d_n$ increase with $n$.
This means that the distribution also changes with $n$ so we 
sometimes write $p$ as $p_n$.

Theorem \ref{thm::core-risk}
shows that the risk over the cores 
where $p(x) > \xi +a$
is exponentially small
as long as we take
$a = C \eta$ for some $C>0$.
The total risk is therefore 
the exponential bound plus the probability that a point fails to satisfy
$p(x) > \xi +a$. Formally:

\begin{corollary}
Assume the conditions of
Theorem \ref{thm::core-risk}.
The cluster risk is bounded by
\begin{equation}
P(p(X) < \xi + C \eta) + e^{-n b}.
\end{equation}
\end{corollary}

Note that, in the corollary, it is not necessary to let $h\to 0$.
To further control the risk beyond the cores, 
we need to make sure that
$P(p(X) < \xi + C \eta)$ is small.
To do this,
especially in the high-dimensional case,
we need to assume that the clusters are well-defined and are well-separated.
We call these assumptions ``low noise'' assumptions
since they are similar in spirit
to the Tsybakov low noise assumption that is often used in 
high-dimensional classification 
\citep{Audibert2007}.
Specifically, we assume that following:

(Low Noise Assumptions:)
\begin{enum}
\item Let $\sigma_n$ be the minimal distance between critical points of $p_n$.
We assume that $\sigma = \liminf_n \sigma_n > 0$.
\item Let $m_n$ be the number of modes of $p_n$.
Then $\limsup_{n\to\infty} m_n < \infty$.
\item $\lim_{n\to \infty}\min_j p_n(m_j) > 0$.
\item $\xi_n \leq n^{-\gamma}$ for some $\gamma>0$ where
$\xi_n = \sup_{x\in D} p_n(x)$ and
$D = \bigcup_j \partial C_j$.
\item For all small $\epsilon$,
$P(p_n(X) < \epsilon) \leq \epsilon^\beta$
where $\beta = \beta_d$ is increasing with $d$.
\end{enum}

Parts 1-3 capture the idea
that the clusters are well-defined.
It is really parts 4 and 5 that
capture the low noise idea.
In particular, 
part 4 says that the density at the cluster boundaries is small.
(See Figure \ref{fig::xi}.)
Part 5 rules out thick tails.
Note that for a multivariate Normal $N(0,\sigma^2 I)$,
we have that, for any fixed small $\epsilon>0$,
$P(p(X) < \epsilon) \leq e^{-d}$
when $\sigma$ is not too large.
So part 5 automatically holds for distributions with Gaussian-like tails.

\begin{figure}
\begin{center}
\includegraphics[scale=.5]{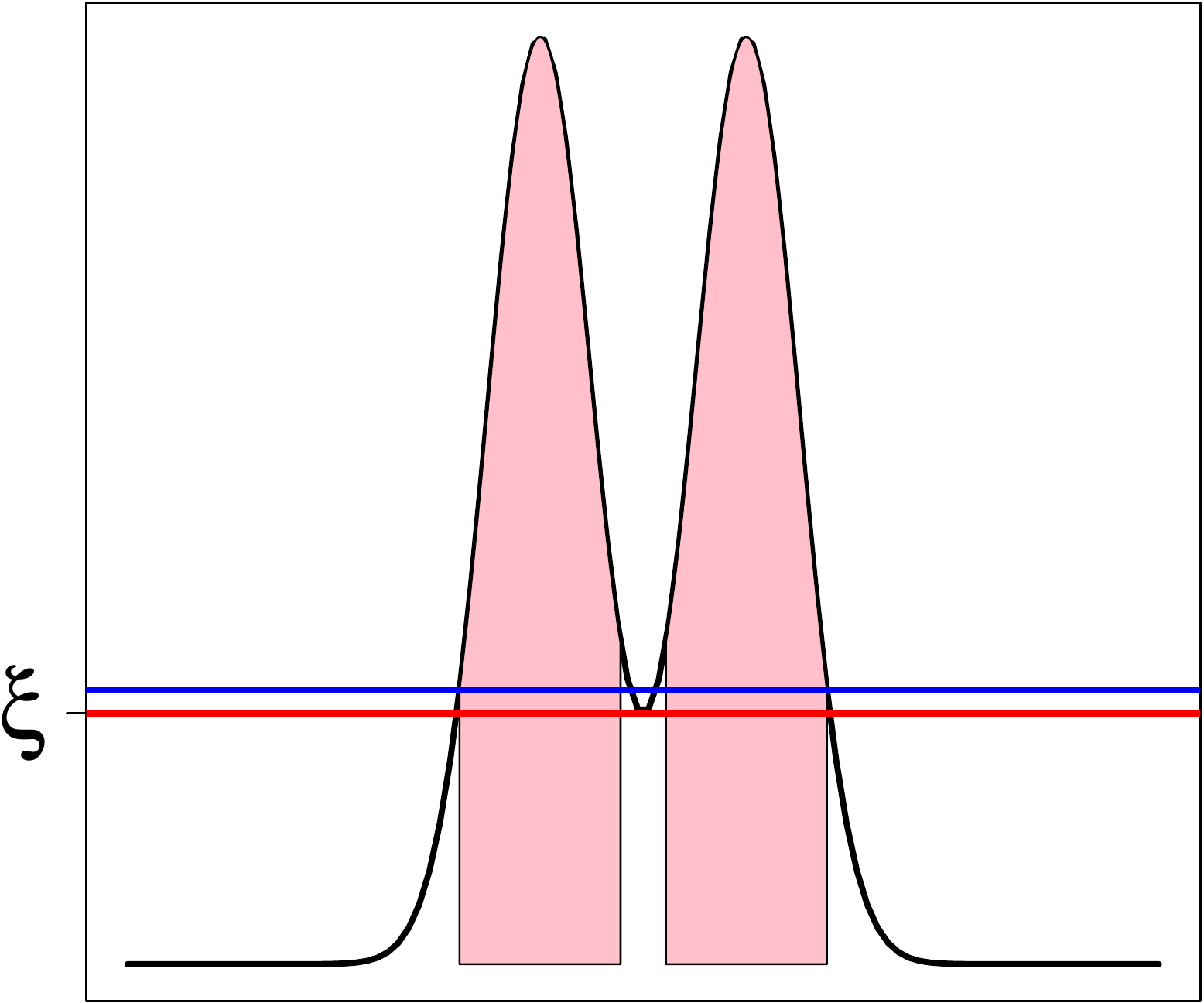}\includegraphics[scale=.5]{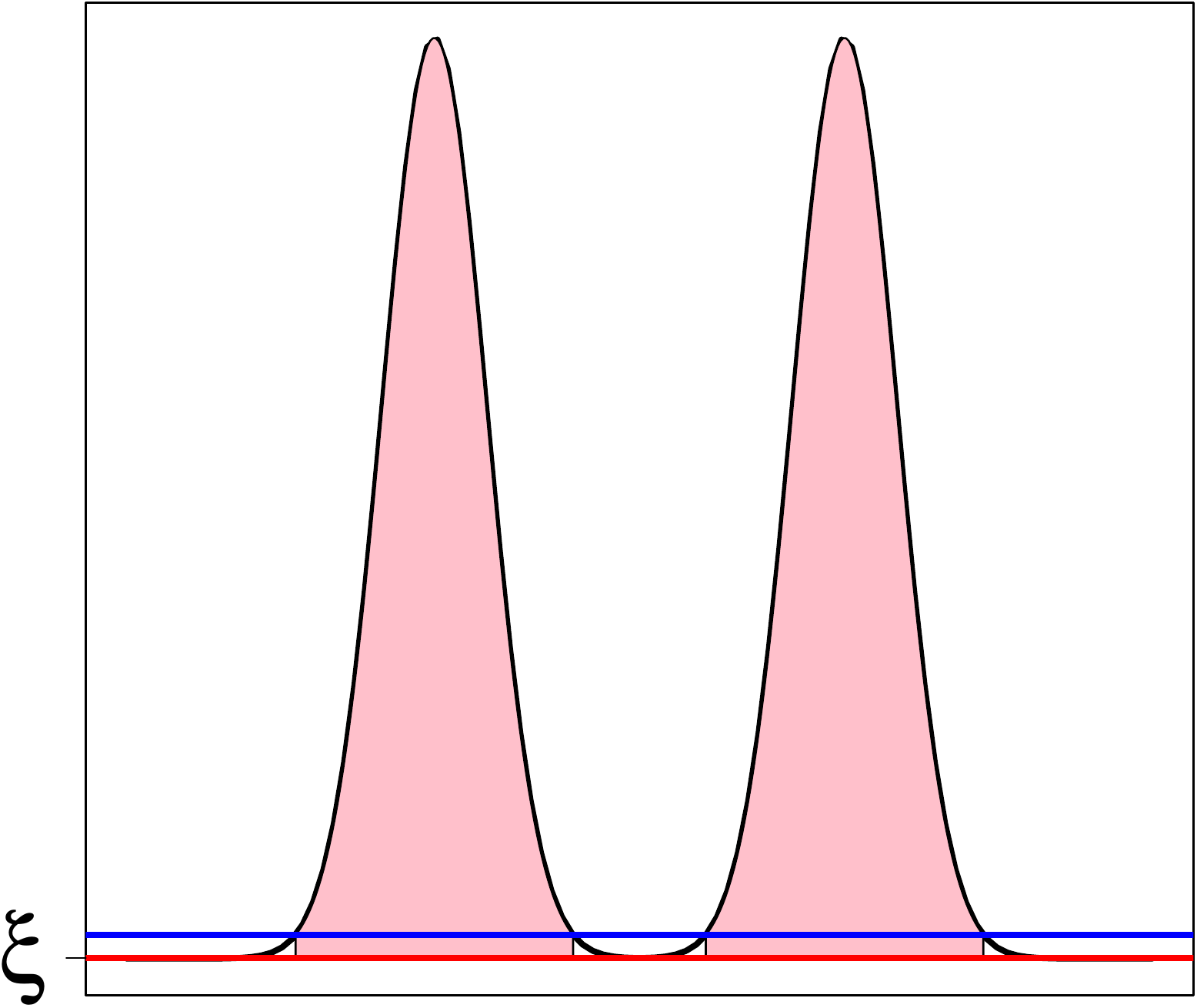}
\end{center}
\caption{\em Left: When clusters are not well separated, $\xi$ is large.
In this case, the mass inside the cluster but outside the core can be large.
Right: When clusters are well separated, $\xi$ is small.
The blue lines correspond to $p(x)=\xi+a$ for $a>0$.
The pink regions are the cluster cores.}
\label{fig::xi}
\end{figure}

\begin{theorem}
Assume that $p_n$ is Morse and that
the low noise conditions hold.
Assume that $p_n$ has three bounded continuous derivatives .
Let $h_n \asymp n^{-1/(5+d)}$. 
Then the clustering risk $R$ satisfies
$$
R \preceq 
\left[\left(\frac{\log n}{n}\right)^{\frac{ \beta}{5+d}} \bigvee 
\left(\frac{1}{n}\right)^{\beta \gamma}\right] + e^{-n b}.
$$
In particular,
$R = O(\sqrt{\log n /n})$
when
$\beta_d \geq \max\{ (d+5)/2,\  1/(2\gamma)\}$.
\end{theorem}

\begin{proof}
For points in the core,
the risk is controlled by
Theorem \ref{thm::core-risk}.
We need now bound the number of pairs outside the cores.
For this, it suffices to bound
$$
\mathbb{P}(p_n(X) < \xi_n + C_g A \eta + 2\eta_0).
$$
For this choice of bandwidth, 
Lemma \ref{lemma::kernel}
implies that
$\eta = O_P(\log n/n^{5+d})$.
From the low noise assumption, the above probability is bounded by
$\xi_n^\beta \vee (\log n/n)^{\beta/(d+5)}$.
\end{proof}

{\bf Remark:}
Parts 4 and 5 of the low noise assumption can be replaced
by a single, slightly weaker assumption, namely,
$P(|p_n(X)-\xi_{n,j}|\leq \epsilon)\leq \epsilon^\beta$
where $\xi_{n,j} = \sup_{x\in\partial {\cal C}_j}p(x)$.
The condition only need hold near the boundaries of the clusters.

\subsection{Gaussian Clusters}

Recently, 
\cite{tan2015statistical}
showed that a type of clustering known as convex clustering
yields the correct clustering with high probability,
even with increasing dimension,
when the data are from a mixture of Gaussians.
They assume that each Gaussian has covariance $\sigma^2 I$
and that the means are separated by a factor of order $\sqrt{d}$.
Here we show a similar result for mode clustering.
The clustering is based on a kernel estimator with a small but fixed
bandwidth $h>0$.

Let
$X_1,\ldots, X_n \sim \sum_{j=1}^k \pi_j N(\mu_j,\sigma^2 I)$
so that $X_i$ has density
$$
p(x) = \sum_{j=1}^k  \frac{\pi_j}{\sigma^d (2\pi)^{d/2}}e^{- ||X-\mu_j||^2/(2\sigma^2)}.
$$

\begin{lemma}
Let
$X\sim p$ and let
$\epsilon >0$.
Suppose that 
\begin{equation}\label{eq::eps-condition}
\epsilon \leq 
\min_j \left( \frac{\pi_j^{1/d}}{\sqrt{2\pi} \sigma e^{16}}\right)^d
\end{equation}
and that
\begin{equation}\label{eq::sep-condition}
\min_{j\neq k}||\mu_j - \mu_k|| >
2\sigma \max_j 
\sqrt{2d \log\left(\frac{1}{\sigma \sqrt{2\pi}}\right) + 
2 \log\left(\frac{1}{\epsilon}\right)-
2\log\left(\frac{1}{\pi_j}\right)}.
\end{equation}
Then
$$
\mathbb{P}(p(X) < \epsilon) \leq e^{-8d}.
$$
\end{lemma}

{\bf Remark:}
Given the condition on $\epsilon$,
we can re-write
(\ref{eq::sep-condition}) as
$$
\min_{j\neq k}||\mu_j - \mu_k|| > C' \sqrt{d}
$$
for a constant $C'>0$.

\begin{proof}
Let 
$$
c = \min_j \sqrt{2 \log \left(\frac{\pi_j}{\epsilon \sigma^d (2\pi)^{d/2}}\right)}
$$
and let
$B_j = \{ x:\ ||x-\mu_j||/\sigma \leq c\}$, $j=1,\ldots, k$.
The sets
$B_1,\ldots, B_k$ are disjoint
due to (\ref{eq::sep-condition}).

First we claim that
$$
p(x) < \epsilon\ \ \ \ {\rm implies\ that}\ \ \ \ x\in \Biggl(\bigcup_s B_s\Biggr)^c = \bigcap_s B_s^c.
$$
To see this, let
$x\in B_j$ for some $j$.
Then, from the definition of $B_j$ and $c$,
$$
p(x) = 
\sum_{s=1}^k  \frac{\pi_s}{\sigma^d (2\pi)^{d/2}}e^{- ||X-\mu_s||^2/(2\sigma^2)}
 \geq
\frac{\pi_j}{\sigma^d (2\pi)^{d/2}}e^{- ||X-\mu_j||^2/(2\sigma^2)}
 \geq
\frac{\pi_j}{\sigma^d (2\pi)^{d/2}}e^{- c^2/2} \geq \epsilon.
$$
That is,
$x\in B_j$ for some $j$ implies
$p(x) \geq \epsilon$ and so the claim follows.

Let $Y\in\{1,\ldots, k\}$
where $P(Y=j) = \pi_j$.
We can write
$X = \sum_j I(Y=j) X_j$
where $X_j \sim N(\mu_j,\sigma^2 I)$.
Of course, $X\stackrel{d}{=} X_j$ when $Y=j$.
Note that
$||X_j-\mu_j||^2/\sigma^2 \sim \chi^2_d$.
Hence,
\begin{align*}
P( p(X) < \epsilon) & \leq
P\left(X \in \bigcap_s B_s^c\right) =
\sum_j \pi_j P\left(X \in \bigcap_s B_s^c\,\Biggm|\, Y=j\right)\\
&= 
\sum_j \pi_j P(X_j \in \bigcap_s B_s^c) \leq
\sum_j \pi_j P(X_j \in B_j^c)\\
&= \sum_j \pi_j P\left( \frac{||X_j - \mu_j||}{\sigma} > c\right)\\
&= \sum_j \pi_j P\left( \chi_d^2 > c^2\right)=
P\left( \chi_d^2 > c^2\right).
\end{align*}
From Lauren and Massart (2000), 
(see also Lemma 11 of Obizinski et al)
when $t\geq 2d$,
$$
P\left( \chi_d^2 > t\right) \leq
\exp\left( - \frac{t}{2} \left(1 - 2 \sqrt{\frac{2d}{t}}\right)\right) 
$$
The last quantity is bounded above by
$e^{-t/4}$ when
$t\geq 32d$.
By the condition on $\epsilon$,
$c^2 \geq 32d$.
Hence
$$
P(p(X) < \epsilon) \leq P\left( \chi^2 > c^2\right) \leq
e^{-c^2/4} \leq e^{-8d}.
$$
\end{proof}

\begin{theorem}
Let
$X_1,\ldots, X_n \sim \sum_{j=1}^k \pi_j N(\mu_j,\sigma^2 I)$.
Let $\hat p_h$
be the kernel density estimator with
fixed bandwidth $h>0$ satisfying
$$
0 < h < \frac{1}{2}  \min_j 
\left( \frac{\pi_j^d}{\sqrt{2\pi} \sigma e^{16}}\right)^d.
$$
Let $D = \bigcup_j \partial{\cal C}_j$ and define
$\Gamma = \min_j  d(\mu_j,D)$.
Suppose that $p$ is Morse,
$$
\Gamma > \sigma \sqrt{32 d + 2 \log\left(\frac{1}{\min_j \pi_j}\right)}
$$
and that
\begin{equation}
\min_{j\neq k}||\mu_j - \mu_k|| >
2\sigma \max_j 
\sqrt{2d \log\left(\frac{1}{\sigma \sqrt{2\pi}}\right) + 
\log\left(\frac{1}{\epsilon}\right)-
2\log\left(\frac{1}{\pi_j}\right)}.
\end{equation}
Then, for all large $n$,
$$
\mathbb{P}\Biggl(\hat c(X_j,X_k) \neq c(X_j,X_k) \ {\rm for\ some\ }j,k\Biggr) \leq e^{-8d} + e^{-nb}.
$$
\end{theorem}

\begin{proof}
By Corollary 5,
the cluster risk is bounded by
$P(p(X)  < \xi + C \eta) + e^{-nb}$.
With a fixed bandwidth not tending to 0,
the bias dominates for all large $n$, and so
$\eta < ch$ for some $c>0$,
except on a set of exponentially small probability.
The condition on $\Gamma$ implies that
$$
\xi = \sup_{x\in D}p(x) \leq 
\min_j \left( \frac{\pi_j^{1/d}}{\sqrt{2\pi} \sigma e^{16}}\right)^d.
$$
So $\epsilon \equiv \xi + C\eta = \xi + ch$ satisfies
(\ref{eq::eps-condition}).
By the previous lemma,
$P(p(X)  < \xi + C h)\leq e^{-8d}$.
\end{proof}

{\bf Remark:}
The theorem implies the following.
As long as the means are separated from each other and from the cluster boundaries
by at least $\sqrt{d}$,
then a kernel estimator 
has cluster risk
$e^{-8d} + e^{-nb}$.
It is not necessary to make the bandwidth tend to 0.

\subsection{Low Dimensional Analysis}

In this section we assume that the dimension $d$ is fixed.
In this case, it is possible to
use a different approach to bound the risk.
We do not make the low noise assumption.
The idea is to use results on the stability of dynamical systems
(Chapter 17 of Hirsch, Smale and Devaney 2004).
As before $p$ is a Morse function and $\tilde p$
is another function.
Define $\eta, \xi, C_g$ and ${\cal C}^\dagger(a)$ as in the previous sections.

Let ${\cal C}$ be a cluster with mode $m$.
Choose a number $a$ such that
\begin{equation}
0 < a < p(m) - A \eta C_g - \xi.
\end{equation}
For any $x$ in the interior of ${\cal C}$,
let
\begin{equation}
t(x) = \inf\Bigl\{ t:\ \pi_x(t)\in {\cal C}^\dagger(a)\Bigr\}.
\end{equation}
If $x\in\partial C$ then
$t(x) = \infty$ since $\pi_x(t)$ converges to a saddlepoint on the boundary.
But for any interior point, $t(x) < \infty$.
For $x\in {\cal C}^\dagger(a)$ we define
$t(x) = 0$.

Our first goal is to control the difference
$||\tilde\pi_x(t(x))- \pi_x(t(x))||$.
And to do this, we first need to bound $t(x)$.
Let
\begin{equation}
\Delta(x) = \inf_{0 \leq t \leq t(x)}||g(\pi_x(t))||.
\end{equation}
Now, $\Delta(x) >0$ for each $x\notin\partial {\cal C}$.
However, as $x$ gets closer to the boundary, 
$\Delta(x)$ approaches 0.
We need an assumption about how fast 
$\Delta(x)$ approaches 0 
as $x$ approaches $\partial {\cal C}$
which is captured in the following assumption:

(B) 
Let ${\cal B}_\delta = \{ x\in {\cal C}:\ d(x,\partial {\cal C}) = \delta\}$.
There exists $\gamma>0$ such that,
for all small $\delta >0$,
\begin{equation}
x\in {\cal B}_\delta \ \ \ {\rm implies\ that\ }\ \ \ \Delta(x) \geq c \delta^\gamma.
\end{equation}

\begin{lemma}
Assume condition B.
If $d(x,\partial {\cal C}) \geq \delta$ then
$t(x) \leq p(m)/\delta^{2\gamma}$.
\end{lemma}

\begin{proof}
Let $z = \pi_{x}(t(x))$ and $x(s) = \pi_x(s)$.
Then, 
\begin{align*}
p(m) \geq p(z) - p(x) &=
\int_0^{t(x)} \frac{\partial p(x(s))}{\partial s} ds =
\int_0^{t(x)} g(x(s))^T \pi_x'(s) ds \\
&= 
\int_0^{t(x)} ||g(x(s))||^2 ds \geq  t(x)\Delta^2(x) = t(x) \delta^{2\gamma}.
\end{align*}
\end{proof}

Now we need the following result which
is Lemma 6 of Arias-Castro et al (2013) adapted from Section 17.5 of 
Hirsch, Smale and Devaney (2004)

\begin{lemma}
Let $\eta_1 = \sup_x ||\nabla p(x) - \nabla \tilde p(x)||$.
For all $t\geq 0$,
\begin{equation}
||\tilde \pi_x(t) - \pi_x(t)|| \leq
\frac{\eta_1}{\kappa_2 \sqrt{d}} e^{\kappa_2 \sqrt{d}\, t}
\end{equation}
where $\kappa_2 = \sup_x||\nabla^2 p(x)||$.
\end{lemma}

We now have the following result.

\begin{theorem}
Let
\begin{equation}\label{eq::thisisdelta}
\delta =
\left(\frac{\kappa_2 \sqrt{d}p(m)}{\log(\kappa_2 \sqrt{d}/\sqrt{\eta_1})}\right)^{\frac{1}{2\gamma}}.
\end{equation}
Let $x,y\in {\cal C}$.
Suppose that
$d(x,\partial {\cal C}) \geq \delta$ and
$d(y,\partial {\cal C}) \geq \delta$.
Also, suppose that $\eta_1 < a^2/C_g$.
Then,
for all small $\eta$,
$\tilde{\rm dest}(x) = \tilde{\rm dest}(y)$.
\end{theorem}

\begin{proof}
We prove the theorem in the following steps:

1. If $x\in C^\dagger(a)$
and $\epsilon < a/C_g$ 
then $B(x,\epsilon)\subset C^\dagger(a')$
where $a' = a - \epsilon C_g >0$.

Proof: Let $y\in B(x,\epsilon)$.
Expanding $p(y)$ around $x$,
$p(y) \geq p(x) - ||y-x|| C_g \geq \xi+a - \epsilon C_g = \xi + a'$.

2. 
Let $t = t(x)$.
If $d(x,\partial {\cal C}) \geq \delta$ then
$\tilde \pi_x(t)\in C^\dagger(a')$
where $a' = a - \sqrt{\eta_1} C_g >0$.

Proof: By definition,
$\pi_x(t) \in C^\dagger(a)$.
From the previous Lemmas,
\begin{align*}
||\tilde \pi_x(t) - \pi_x(t)|| & \leq
\frac{\eta_1}{\kappa_2 \sqrt{d}} e^{\kappa_2 \sqrt{d}\ t}\\
& \leq
\frac{\eta_1}{\kappa_2 \sqrt{d}} e^{\kappa_2 \sqrt{d}p(m) \delta^{-2\gamma}} \leq \sqrt{\eta_1}.
\end{align*}
The last inequality follows from the definition of $\delta$.
Hence,
$\tilde \pi_x(t)\in B(\pi_x(t),\sqrt{\eta_1})$.
It follows from part 1 that
$\tilde \pi_x(t)\in C^\dagger(a')$.
The fact that $a' >0$ follows from the fact that
$\eta_1 < a^2/C_g$.

3. From Theorem 3,
$\tilde p$ has a mode $\tilde m$ in
$C^\dagger(a')$ and has no other critical points in 
$C^\dagger(a')$.
So, letting $z = \tilde \pi_x(t)$,
and noting that $z$ is on the path $\tilde \pi_x(t)$,
$$
\tilde {\rm dest}(x) = \lim_{s\to \infty} \tilde \pi_x(s) = 
\lim_{s\to \infty} \tilde \pi_z(s) = \tilde m.
$$

Applying steps 1, 2 and 3 to $y$ we have that
$\tilde {\rm dest}(y)$ also equals $\tilde m$.
\end{proof}

Now let $\hat p_h$ be the kernel density estimator with $h=h_n \asymp n^{-1/(5+d)}$.
In this case $\eta_1 = O_P(n^{-2/(6+d)})$ so,
with $\delta$ defined as in (\ref{eq::thisisdelta}), we have
$$
\delta \equiv \delta_n \asymp (\log n)^{-1/(2\gamma)}.
$$
By the previous theorem,
there are no clustering errors for
data points $X_i$ such that $d(X_i, D) \succeq \delta_n$
where $D = \bigcup_j \partial {\cal C}_j$
as long as
$\eta_1 = \sup_x ||\nabla p(x) - \hat\nabla p(x)|| < a^2/C_g$
which holds except on a set of exponentially small probability (Lemma 1).
Hence, we have:

\begin{corollary}
Assume that $p$ is a Morse function with
finitely many critical values.
Denote the modes and clusters by
$m_1,\ldots, m_k$ and
${\cal C}_1,\ldots, {\cal C}_k$.
Suppose that condition (B) holds in each cluster.
Let $\hat p_h$
be the kernel density estimator.
Let $\eta = \max\{\eta_0,\eta_1,\eta_2\}$
where
$$
\eta_0 = \sup_x|\hat p_h(x) - p(x)|,\ 
\eta_1 = \sup_x|| \nabla \hat p_h(x) - \nabla p(x)||,\ 
\eta_2 = \sup_x|| \nabla^2 \hat p_h(x) - \nabla^2 p(x)||.
$$
Let $D = \bigcup_j \partial {\cal C}_j$ and let
$$
{\cal X} = \{ X_i:\ d(X_i,D) \geq \delta_n\}
$$
where
$$
\delta_n =
\left(\frac{\kappa_2 \sqrt{d}p(m)}{\log(\kappa_2 \sqrt{d}/\sqrt{\eta_1})}\right)^{\frac{1}{2\gamma}} \asymp
 (\log n)^{-1/(2\gamma)}.
$$
If $h_n\to 0$ and $n h_n^{d+4} \to \infty$, then
\begin{equation}
\mathbb{P}\Biggl( \hat c(X_i,X_j) \neq c(X_i,X_j)\ \ 
{\rm for\ any\ }\ X_i,X_j \in {\cal X}\Biggr) \leq e^{-n b}
\end{equation}
for some $b>0$.
\end{corollary}

Thus, the clustering risk is exponentially small
if we exclude points that are close to the boundary.

\section{Experiments}

An example of highly non-spherical mode clusters in two dimensions is given in Figure~\ref{fig::experiment}, left panel.
The true density (contours shown in blue) has two modes, with the corresponding basins of attraction shown in blue and green.
Mean shift (using a Gaussian kernel with bandwidth $1$) is applied to the $1000$ points sampled from the density as plotted, and all but the points shown in red are correctly clustered.
All but $1\%$ of points are correctly clustered, despite a total variation distance of about $0.29$ between the true and estimated densities.

Our theoretical results show that mean shift clustering should perform well even in high dimensions, assuming the bulk of the basins of attraction are well-separated by low density regions.
We simulate such a setting in $10$ dimensions, were we measure the performance of mean shift clustering on samples drawn from a mixture of two equal weight Gaussian components.
The norm of the difference between the means is $5$, and each component has randomly generated non-spherical covariance matrix with eigenvalues between $0.5$ and $2$.
The center panel of Figure~\ref{fig::experiment} shows the average clustering error as a function of the sample size $n$ and bandwidth $h$, after $75$ replications of the procedure.
With only $50$ samples, an average error of $0.05$ is achieved with the appropriate bandwidth.

The effect of component separation is demonstrated further in the right panel of Figure~\ref{fig::experiment}.
Here, we draw $n=300$ samples from an equal weight mixture of 
two unit covariance Gaussians in 
two dimensions, and measure the clustering error of mean shift (averaged over $35$ replications).

\begin{figure}
\begin{center}
\includegraphics[width=0.3\textwidth]{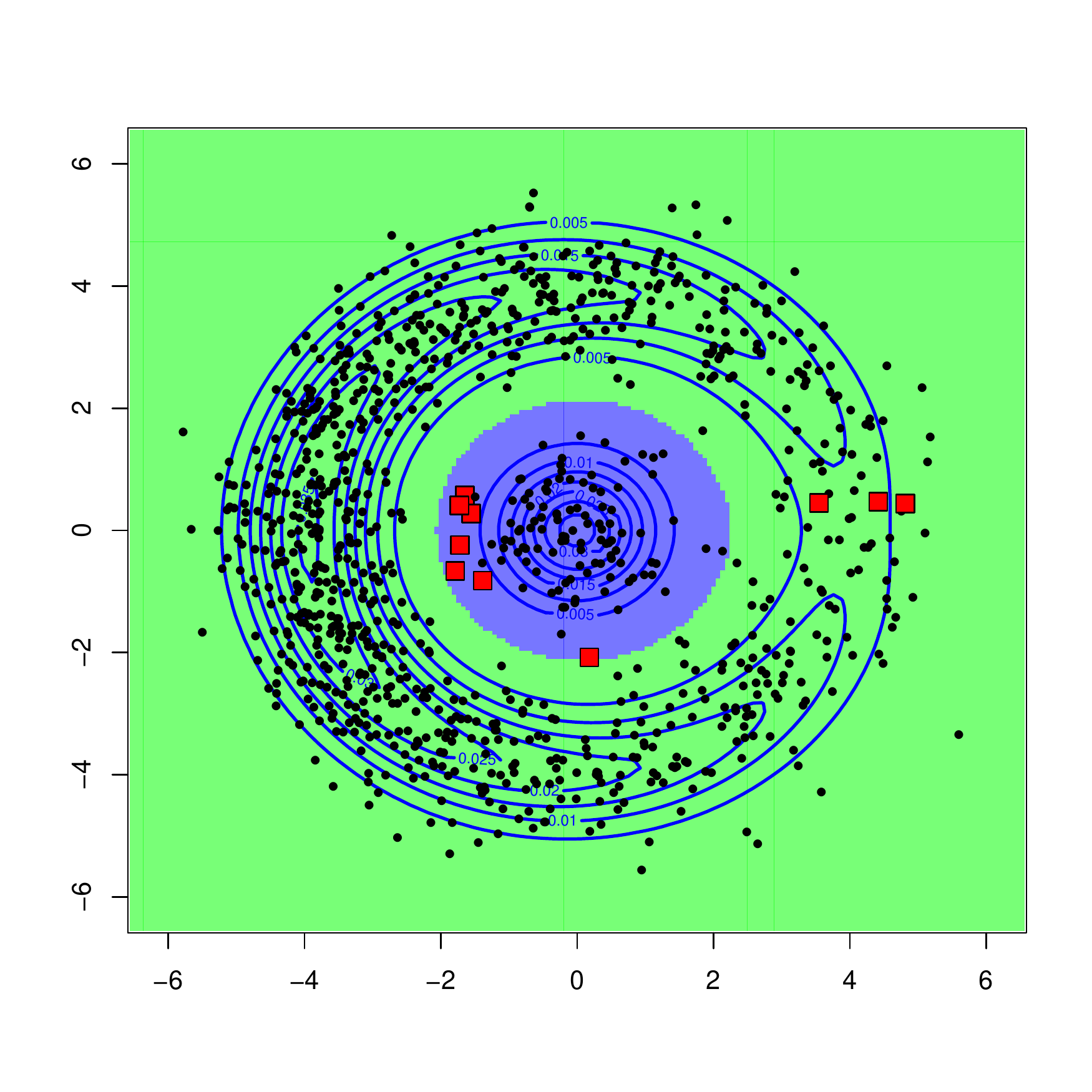} ~
\includegraphics[width=0.3\textwidth]{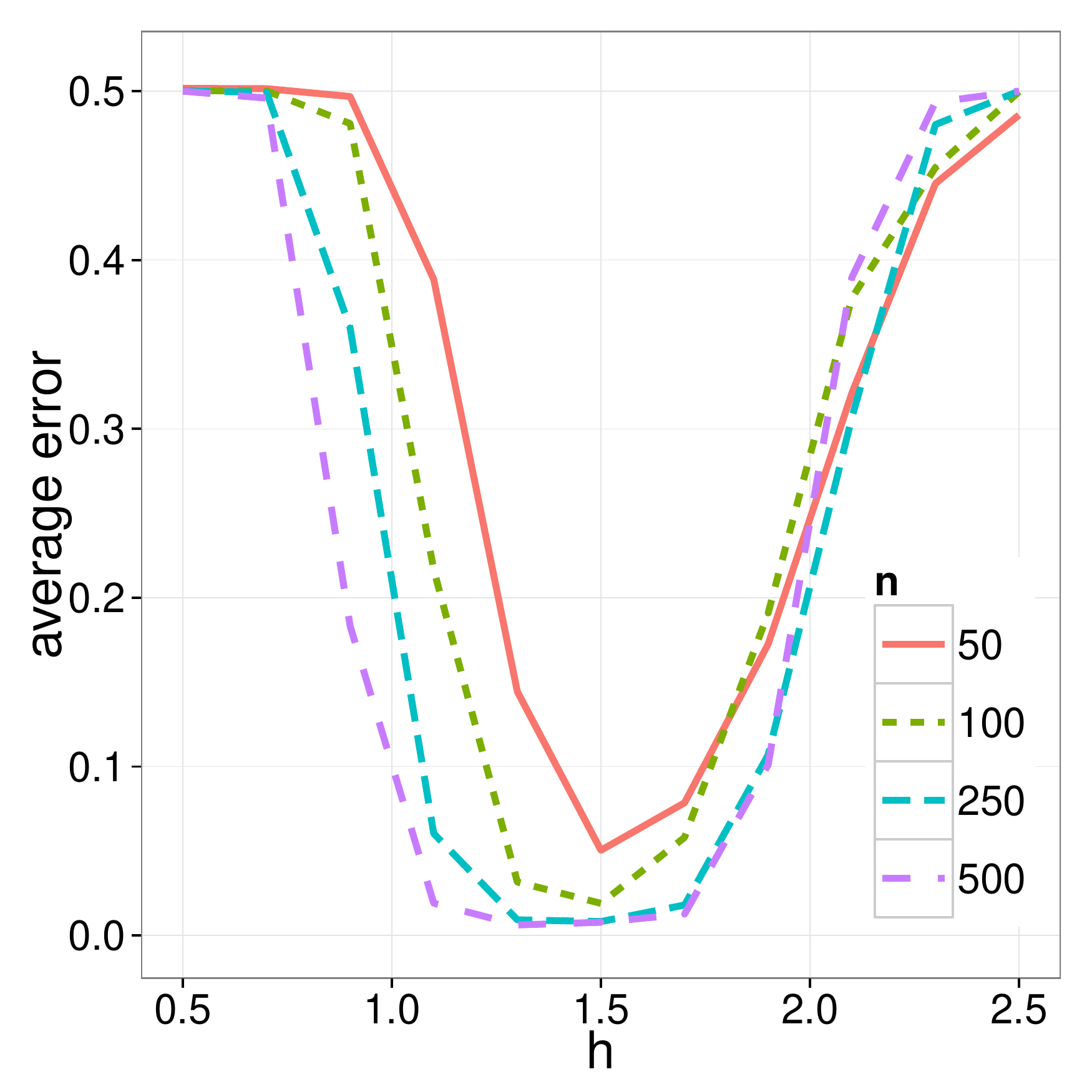} ~
\includegraphics[width=0.3\textwidth]{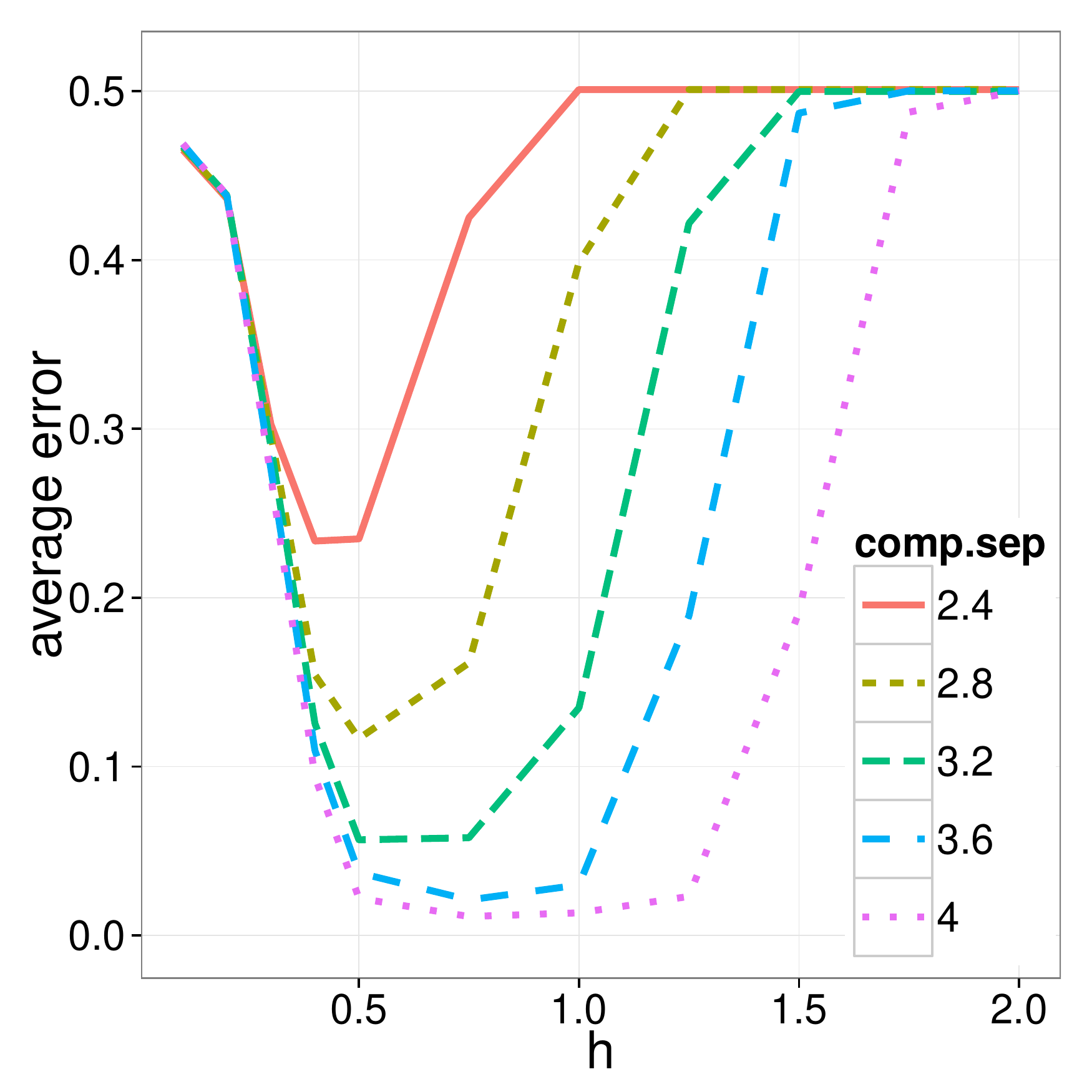} 
\end{center}
\caption{\em Left: example of highly non-convex basins of attraction.
Center: small sample complexity in high dimensions due to well-separated clusters.
Right: effect of cluster separation, ranging from nearly unimodal to having two well-separated modes.}
\label{fig::experiment}
\end{figure}

\section{Conclusion}

Density mode clustering --- also called mean-shift clustering ---
is very popular in certain fields such as computer vision.
In statistics and machine learning it is
much less well known.
This is too bad because it is a simple, nonparametric and very general
clustering method.
And as we have seen, it is not necessary to estimate the density well
to get a small clustering risk.
Because of this, mode clustering can be effective even in high dimensions.

We have developed a bound on the pairwise risk
of density mode clustering.
The risk within the cluster cores --- the high density regions ---
is very small with virtually no assumptions.
If the clusters are well-separated (low noise condition)
then the overall risk is small,
even in high dimensions.

Several open questions remain such as:
how to estimate the risk, how to choose a good bandwidth
and what to do when the low noise condition fails.
Regarding the last point,
we believe it should be possible to
identify regions where the low noise conditions fail.
These are essentially parts of the cluster boundaries with non-trivial mass.
In that case, there are two ways to reduce the risk.
One is to merge poorly separated clusters.
Another is to allow ambiguous points to be assigned to more than one cluster.
For research in this direction, see
\cite{li2007nonparametric, Chen2014GMRE}.

\acks{We would like to acknowledge support for this project
from the National Science Foundation (NSF grant XXXX) and XXXX.}

\bibliography{paper.bib}

\end{document}